\documentclass[12pt]{amsart}
\usepackage{amssymb}
\usepackage{mathabx}
\usepackage{xypic}
\usepackage{hhline}
\usepackage{dcpic}
\usepackage{hyperref}
\usepackage{graphicx}
\usepackage{caption}
\usepackage{subfig}
\usepackage[section]{placeins}
\graphicspath{{Figure/}}

\usepackage{booktabs}
\usepackage{multirow}

\newtheorem{theorem}{Theorem}

\newtheorem{conjecture}[theorem]{Conjecture}

\newtheorem{proposition}[theorem]{Proposition}

\newtheorem{remark}[theorem]{Remark}

\numberwithin{theorem}{section}
\numberwithin{equation}{section}

\begin{document}


\makeatletter
\def\Ddots{\mathinner{\mkern1mu\raise\p@
\vbox{\kern7\p@\hbox{.}}\mkern2mu
\raise4\p@\hbox{.}\mkern2mu\raise7\p@\hbox{.}\mkern1mu}}
\makeatother

\newcommand{\OP}[1]{\operatorname{#1}}
\newcommand{\GO}{\OP{GO}}
\newcommand{\leftexp}[2]{{\vphantom{#2}}^{#1}{#2}}
\newcommand{\leftsub}[2]{{\vphantom{#2}}_{#1}{#2}}
\newcommand{\rightexp}[2]{{{#1}}^{#2}}
\newcommand{\rightsub}[2]{{{#1}}_{#2}}
\newcommand{\AI}{\OP{AI}}
\newcommand{\gen}{\OP{gen}}
\newcommand{\prim}{\OP{\star prim}}
\newcommand{\Image}{\OP{Im}}
\newcommand{\Spec}{\OP{Spec}}
\newcommand{\Ad}{\OP{Ad}}
\newcommand{\tr}{\OP{tr}}
\newcommand{\spec}{\OP{spec}}
\newcommand{\scopy}{\OP{end}}
\newcommand{\ord}{\OP{ord}}
\newcommand{\Cent}{\OP{Cent}}
\newcommand{\wellip}{\OP{w-ell}}
\newcommand{\Nrd}{\OP{Nrd}}
\newcommand{\error}{\OP{RelErr}}
\newcommand{\LCM}{\OP{LCM}}
\newcommand{\sgn}{\OP{sgn}}
\newcommand{\SU}{\OP{SU}}
\newcommand{\Hom}{\OP{Hom}}
\newcommand{\Inter}{\OP{Int}}
\newcommand{\diag}{\OP{diag}}
\newcommand{\Sym}{\OP{Sym}}
\newcommand{\GSp}{\OP{GSp}}
\newcommand{\GL}{\OP{GL}}
\newcommand{\GSO}{\OP{GSO}}
\newcommand{\height}{\OP{ht}}
\newcommand{\erf}{\OP{erf}}
\newcommand{\vol}{\OP{vol}}
\newcommand{\cusp}{\OP{cusp,\tau}}
\newcommand{\un}{\OP{un}}
\newcommand{\disci}{\OP{disc,\tau_{\it i}}}
\newcommand{\cuspi}{\OP{cusp,\tau_{\it i}}}
\newcommand{\ellip}{\OP{ell}}
\newcommand{\sph}{\OP{sph}}
\newcommand{\gsimp}{\OP{sim-gen}}
\newcommand{\Aut}{\OP{Aut}}
\newcommand{\disc}{\OP{disc,\tau}}
\newcommand{\sdisc}{\OP{s-disc}}
\newcommand{\aut}{\OP{aut}}
\newcommand{\End}{\OP{End}}
\newcommand{\barQ}{\OP{\overline{\mathbf{Q}}}}
\newcommand{\barQp}{\OP{\overline{\mathbf{Q}}_{\it p}}}
\newcommand{\Gal}{\OP{Gal}}
\newcommand{\PGL}{\OP{PGL}}
\newcommand{\simp}{\OP{sim}}
\newcommand{\pri}{\OP{prim}}
\newcommand{\Normal}{\OP{Norm}}
\newcommand{\Ind}{\OP{Ind}}
\newcommand{\St}{\OP{St}}
\newcommand{\unit}{\OP{unit}}
\newcommand{\reg}{\OP{reg}}
\newcommand{\SL}{\OP{SL}}
\newcommand{\Frob}{\OP{Frob}}
\newcommand{\Id}{\OP{Id}}
\newcommand{\GSpin}{\OP{GSpin}}
\newcommand{\Norm}{\OP{Norm}}


\title[Pseudorandomness of Sato-Tate Distributions]{Pseudorandomness of Sato-Tate Distributions for Elliptic Curves}
\author{Chung Pang Mok, Huimin Zheng}

\address{School of Mathematical Sciences, Soochow University, 1 Shi-Zi Street, Suzhou 215006, Jiangsu Province, China}

\email{zpmo@suda.edu.cn}

\address{Department of Mathematics, Nanjing University, 22 Hankou Road, Nanjing 210093, Jiangsu Province, China}

\address{Jiangsu National Center for Applied Mathematics}
\email{zhhm@smail.nju.edu.cn}

\maketitle

\begin{abstract}
In this paper we propose conjectures that assert that, the sequence of Frobenius angles of a given elliptic curve over $\mathbf{Q}$ without complex multiplication is pseudorandom, in other words that the Frobenius angles are {\it statistically independently} distributed with respect to the Sato-Tate measure. Numerical evidences are presented to support the conjectures. 
\end{abstract}

\tableofcontents

\section{Introduction and Statement of Conjectures}

Let $E$ be an elliptic curve defined over the field $\mathbf{Q}$ of rational numbers. Assume that $E$ does not have complex multiplication. Denote by $N$ the conductor of $E$, so that $E$ has good reduction at primes $p$ not dividing $N$. Let $\{p_k\}_{k \geq 1}$ be the set of primes in ascending order. For $p_k$ not dividing $N$, define as usual the quantity $a_{p_k}:= p_k+1-\#E(\mathbf{F}_{p_k})$ (here $\mathbf{F}_{p_k}$ is the finite field of cardinality $p_k$, and $E(\mathbf{F}_{p_k})$ is the group of points of $E$ over $\mathbf{F}_{p_k}$). By the Hasse bound one has $|a_{p_k}| \leq 2 p_k^{1/2}$. We define $x_k \in [0,1]$ for $k \geq 1$, referred to as the (normalized) Frobenius angle of $E$ at the prime $p_k$, by the condition:
\[
a_{p_k} = 2p_k^{1/2} \cos(\pi x_k)
\]
if $p_k$ does not divide $N$, and we simply define $x_k=1/2$ if $p_k$ divides $N$. The Sato-Tate conjecture, as established by Taylor {\it et. al.} \cite{CHT,T,HSBT,BLGHT}, states that the sequence $\{x_k\}_{k \geq 1}$ is uniformly distributed with respect to the Sato-Tate measure on $[0,1]$. Specifically define the Sato-Tate measure $\mu_{ST}$ on $[0,1]$ (which is a probability measure) by:
\[
d \mu_{ST} = 2 \sin^2(\pi u) \, du
\] 
(where $du$ is the Lebesgue measure on $[0,1]$) and for $x_k$ as above, denote by $\delta_{x_k}$ the Dirac point mass distribution on $[0,1]$ supported at $x_k$. For any integer $K \geq 1$ consider the probability distribution on $[0,1]$ given by:
\begin{eqnarray}
\frac{1}{K} \sum_{k=1}^K \delta_{x_k}
\end{eqnarray}
then the Sato-Tate conjecture states that the probability distribution (1.1) converges weakly to the Sato-Tate measure $\mu_{ST}$ as $K$ tends to infinity.

\bigskip
In this paper we propose the following refinement of the original Sato-Tate conjecture, that in addition to uniform distribution, we conjecture that the sequence $\{x_k\}_{k \geq 1}$ is pseudorandom, in other words the $x_k$'s are in fact {\it statistically independently} distributed with respect to the Sato-Tate measure. Specifically for any integer $s \geq 1$, denote by $\mu_{ST}^{[s]}$ the probability measure on $[0,1]^s$ given by the product of $s$ (independent) copies of the original one dimensional Sato-Tate measure $\mu_{ST}$. We consider the joint distributions for $s$ successive terms from the sequence $\{x_k\}_{k \geq 1}$. Thus define $X_k \in [0,1]^s$ for $k \geq 1$ to be the $s$-dimensional vector given by:
\[
X_k = (x_k,x_{k+1}, \cdots,x_{k+s-1})
\]
and denote by $\delta_{X_k}$ the Dirac point mass distribution on $[0,1]^s$ that is supported at $X_k$.

\bigskip
For integer $K \geq 1$ consider similarly the probability distribution on $[0,1]^s$ given by:
\begin{eqnarray}
\frac{1}{K} \sum_{k=1}^K \delta_{X_k}
\end{eqnarray}

\bigskip
We propose the following:
\begin{conjecture}
For any integer $s \geq 1$, the sequence $\{X_k\}_{k \geq 1}$ is uniformly distributed with respect to $\mu_{ST}^{[s]}$. In other words the probability distribution on $[0,1]^s$ as given by (1.2), converges weakly to $\mu_{ST}^{[s]}$, as $K$ tends to infinity. 
\end{conjecture}

\bigskip
For $s=1$ this is the original Sato-Tate conjecture (which is already proved). For $s \geq 2$ this asserts the statistical independence of the distribution of $s$ successive terms from the sequence $\{x_k\}_{k \geq 1}$ with respect to Sato-Tate measure. Thus in terms of statistical distribution, Conjecture 1.1 asserts that, with respect to the Sato-Tate measure, the sequence $\{x_k\}_{k \geq 1}$ is $\infty$-distributed in the sense of Knuth, {\it c.f.} Definition C on page 151 of \cite{K}. It is in this sense that we say that the sequence $\{x_k\}_{k \geq 1}$ is pseudorandom (with respect to the Sato-Tate measure). We present numerical evidences for Conjecture 1.1 in section 2 below.

\bigskip
\begin{remark}
\end{remark}
\noindent Let $E$ and $E^{\prime}$ be elliptic curves over $\mathbf{Q}$ without complex multiplication, and assume that $E$ and $E^{\prime}$ are non-isogenous. Let $\{x_k\}_{k \geq 1}$ and $\{x_k^{\prime} \}_{k \geq 1}$ be the sequences of Frobenius angles associated to $E$ and $E^{\prime}$ respectively. Consider the two dimensional vectors $(x_k,x_k^{\prime}) \in [0,1]^2 $ for $k \geq 1$. Harris \cite{H} established that the sequence $\{ (x_k,x_k^{\prime})\}_{k \geq 1}$ is uniformly distributed with respect to $\mu_{ST}^{[2]}$. By contrast, the setting of our Conjecture 1.1 concerns the statistical independence of the distribution of Frobenius angles for a single elliptic curve.
\bigskip
\bigskip

Even more optimistically, we propose the following quantitive refinement of Conjecture 1.1. Firstly define the extreme discrepancy $D^{[s]}_K$ (with respect to the measure $\mu_{ST}^{[s]}$) for integer $K \geq 1$ as follows. For any rectangular region $\mathcal{R} \subset [0,1]^s$ of the form:
\[
\mathcal{R} = [a_1,b_1) \times \cdots \times [a_s,b_s)
\]
define
\[
A(\mathcal{R};K) = \#\{ 1\leq k\leq K \,\ | \,\ X_k \in \mathcal{R}       \}
\]
\bigskip
Then define
\[
D^{[s]}_K = \sup_{\mathcal{R} } \Big|    \frac{A(\mathcal{R};K)}{K}        -  \mu_{ST}^{[s]}  (\mathcal{R})   \Big|
\]
where $\mathcal{R}$ ranges over all rectangular regions in $[0,1]^s$ as above.

\bigskip

In general for a point $W =(w^{(1)},\cdots,w^{(s)}) \in [0,1]^s$, we denote by $\mathcal{R}_W \subset [0,1]^s$ the rectangular region:
\[
\mathcal{R}_W = [0,w^{(1)} ) \times \cdots \times  [0,w^{(s)} ) 
\]

\bigskip

We define the star discrepancy $D_K^{*,[s]}$ (with respect to the measure $\mu_{ST}^{[s]}$) as:
\[
D^{*,[s]}_K = \sup_{W \in [0,1]^s } \Big|    \frac{A(\mathcal{R}_W;K )}{K}        -  \mu_{ST}^{[s]}  (\mathcal{R}_W)   \Big|
\]

\bigskip
We have $0 \leq D_K^{[s]} ,D_K^{*,[s]}\leq 1$, and the inequalities ({\it c.f.} p. 93 of \cite{KN}):
\begin{eqnarray}
D_K^{*,[s]} \leq D_K^{[s]} \leq 2^s \cdot D_K^{*,[s]}
\end{eqnarray}
(remark that in {\it loc. cit.} the notion of discrepancy with respect to the Lebesgue measure on $[0,1]^s$ is considered, but the same considerations apply verbatim with respect to the measure $\mu_{ST}^{[s]}$ as well).

\bigskip
The discrepancies $D^{[s]}_K $ and $D_K^{*,[s]}$ quantify the uniformity of distribution of the finite set $\{ X_k \}_{k=1}^K$ with respect to the measure $\mu_{ST}^{[s]}$ on $[0,1]^s$.

\bigskip
\begin{conjecture}
For any integer $s \geq 1$ and $\epsilon >0$, there exists a constant $C=C(E,s,\epsilon)$ (depending only on the elliptic curve $E$, $s$ and $\epsilon$), such that:
\[
D^{[s]}_K \leq C K^{\epsilon- \frac{1}{2}}
\]
for any integer $K \geq 1$.
\end{conjecture}

\bigskip
Of course, Conjecture 1.3 is interesting only when $ \epsilon < \frac{1}{2} $; in addition it can also be stated equivalently in terms of the star discrepancy $D_K^{*,[s]}$ instead of $D_K^{[s]}$, by virtue of the inequalities (1.3).

\bigskip

When $s=1$, Conjecture 1.3 was originally formulated by Akiyama-Tanigawa (Conjecture 1 of \cite{AT}), which refines the original Sato-Tate conjecture. In general Conjecture 1.3 is a refinement of Conjecture 1.1; namely that by standard results on uniform distribution ({\it c.f.} p. 93 of \cite{KN}), Conjecture 1.1 is equivalent to the assertion:
\[
\lim_{K \rightarrow \infty} D^{[s]}_K =0.
\]

\bigskip

 Conjecture 1.3 is a very strong statement. Indeed Akiyama-Tanigawa showed that their conjecture (i.e. Conjecture 1.3 in the case $s=1$) implies the validity of the Riemann Hypothesis for the $L$-function associated to $E$ (that the $L$-function associated to $E$ has analytic continuation is of course the consequence of the modularity of $E$); more generally their conjecture implies the validity of the Riemann Hypothesis for all the higher symmetric power $L$-functions associated to $E$, {\it c.f.} Poposition 3.5 of \cite{M} for the precise statement (in {\it loc. cit.} suitable analytic hypotheses on the higher symmetric power $L$-functions are assumed, which are in any case consequences of the Langlands Functoriality Conjecture with respect to the symmetric power functorial liftings of the modular form associated to $E$. The existence of the symmetric power functorial liftings of the modular form associated to $E$ is established by Newton-Thorne \cite{NT} under quite general conditions, including all semistable $E$ for instance). 
 
 \bigskip
 Conversely Nagoshi showed (see Theorem 2 of \cite{N}) that the conjecture of Akiyama-Tanigawa holds (at least) for $ \epsilon > 1/4$, if one supposes the validity of the Riemann Hypothesis for all the higher symmetric power $L$-functions associated to $E$ (again assuming suitable analytic hypotheses on the higher symmetric power $L$-functions); see also \cite{RT} for the explicit version. At this moment we do not know whether our Conjecture 1.3 for $s \geq 2$ can be approached using the theory of $L$-functions. Nevertheless we present numerical evidences for Conjecture 1.3 in section 3 below.

\bigskip

In this paper the computations of the orders of the group of points of elliptic curves over finite fields were performed using the {\it GP/PARI} program. The rest of the computations were then performed using {\it Mathematica 9.0}.

\section{Numerical Evidences for Conjecture 1.1}

With the setting as in Conjecture 1.1, for continuous function $f$ defined on $[0,1]^s$, we would like to test whether:
\[
 \frac{1}{K} \sum_{k=1}^K f(X_k)  \stackrel{?}{\rightarrow} \int_{[0,1]^s} f \, d \mu_{ST}^{[s]}
\]
 as $K$ tends to infinity. 

\bigskip
We consider the following six elliptic curves over $\mathbf{Q}$ without complex multiplication, taken from the {\it $L$-functions and modular forms database}, whose affine Weierstrass equations are given as follows, with conductor $N$ and Mordell-Weil rank $r$ as indicated:

\[
E_1: y^2+y=x^3-x^2, \,\ N=11,\,\ r=0
\]

\[
E_2: y^2+y =x^3-x, \,\ N=37, \,\ r=1
\]

\[
E_3: y^2 +xy=x^3+1, \,\ N=433, \,\ r=2
\]

\[
E_4:y^2+y=x^3-7x+6, \,\ N=5077, \,\ r=3
\]

\[
E_5: y^2+y=x^3-7x+36, \,\ N=545723, \,\ r=4
\]

\[
E_6: y^2+y=x^3-79x+342, \,\ N=19047851, \,\ r=5
\]

\bigskip

\bigskip
We denote the coordinates on $[0,1]^s$ as $u^{(1)},\cdots,u^{(s)}$. For testing statistical independence, it is good enough to choose test functions $f$ of the form:
\[
f(u^{(1)},\cdots,u^{(s)}) = \prod_{i=1}^s f_i(u^{(i)})
\]
for continuous functions $f_i$ on $[0,1]$, in which case we have
\begin{eqnarray}
\int_{[0,1]^s} f \, d\mu_{ST}^{[s]} = \prod_{i=1}^s \int_{[0,1]}  f_i (u^{(i)})  \cdot 2 \sin^2(\pi u^{(i)}) \, d u^{(i)}
\end{eqnarray}

\bigskip
We first consider the case $s=10$. Define the function $f^{[10]}$ on $[0,1]^{10}$:

\begin{eqnarray*}
& & f^{[10]}(u^{(1)},\cdots,u^{(10)}) \\ 
&= & \ln(2 + u^{(1)}) \cdot \ln (3 + u^{(2)}) \cdot \exp(-u^{(3)}) \cdot (1+ u^{(4)})^2 \cdot   (2+ u^{(5)})  \cdot \\
& &  \sqrt{2 + u^{(6)}} \cdot \sqrt{ 3 + u^{(7)}} \cdot (4+ u^{(8)})^{\frac{1}{3}} \cdot  (8+ u^{(9)})^{\frac{1}{4}}  \cdot  \exp( \sqrt{1+u^{(10)}}   )
\end{eqnarray*}

\bigskip
Using the command NIntegrate of {\it Mathematica}, the numerical value of the integral $\int_{[0,1]^{10}}  f^{[10]} \, d\mu_{ST}^{[10]}$ is computed as in (2.1) to be:
\[ 
\int_{[0,1]^{10}}  f^{[10]} \, d\mu_{ST}^{[10]} \stackrel{.}{=}114.076
\]

The numerical results for $\frac{1}{K} \sum_{k=1}^K f^{[10]}(X_k)$ for $K=5000$, $K=10000$, $K=20000$, $K=50000$, and $K=100000$ are tabulated in Figure 1 below.

\begin{figure}[hp] 
	
	\centering  
	
	\begin{tabular}{c|c|c|c|c|c}

		\toprule 
		& $K=5000$  &  $   K=10000$   &  $   K=20000$   &   $K=50000 $ & $K=100000$   \\
		
		\midrule 
        $E_1$    &    113.87      &  113.753        &    113.903      &  114.009     &  114.032\\    
        $E_2$   & 114.196       &  114.08        &  114.074            &   114.128     & 114.181\\

        $E_3$ &    114.534     &  114.576        &       114.493          &     114.154  & 114.237 \\

        $E_4$ &   115.375      &     115.011  &     114.683        &   114.441      &  114.248\\

        $E_5$ & 116.127         &    115.137   &      114.499       &    114.62        & 114.474 \\
        $E_6$ &  116.559      & 115.371      &         115.312     &    114.471       &114.519\\
       
		\bottomrule 
		
	\end{tabular}
    \caption{Numerical results for $\frac{1}{K} \sum_{k=1}^K f^{[10]}(X_k)$}  
    \label{Numres}
\end{figure}

\bigskip
We next consider examples with larger values of $s$. For $s \geq 1$ define the function $g^{[s]}$ on $[0,1]^s$ given by:
\begin{eqnarray*}
 g^{[s]}(u^{(1)},\cdots,u^{(s)}) =  100 \cdot \prod_{i=1}^s \exp ( - u^{(i)} /i)
\end{eqnarray*}

\bigskip
We have:
\begin{eqnarray}
& & \int_{[0,1]^s} g^{[s]} \, d \mu_{ST}^{[s]}  \\ &=& 100 \cdot \prod_{i=1}^s \int_{[0,1]} \exp (-u^{(i)}/i) \cdot 2 \sin^2(\pi u^{(i)})\, d u^{(i)}  \nonumber \\
  &=& 100 \cdot  \prod_{i=1}^s    \left( \big(   1 - \exp (-1/i)  \big)
         \cdot \frac{4 \pi^2 i^3}{1+ 4\pi^2 i^2} \right) \nonumber
\end{eqnarray}

\bigskip
The numerical values of $\int_{[0,1]^s} g^{[s]} \, d \mu_{ST}^{[s]}$ for $s=500$, $s=1000$, and $s=2000$, are computed as in (2.2) to be:

\begin{eqnarray*}
\int_{[0,1]^{500}} g^{[500]} \, d \mu_{ST}^{[500]} \stackrel{.}{=}
3.44034
\end{eqnarray*}
\begin{eqnarray*}
\int_{[0,1]^{1000}} g^{[1000]} \, d \mu_{ST}^{[1000]} \stackrel{.}{=}
2.43333
\end{eqnarray*}
\begin{eqnarray*}
\int_{[0,1]^{2000}} g^{[2000]} \, d \mu_{ST}^{[2000]} \stackrel{.}{=}
1.72086
\end{eqnarray*}

\bigskip
While the numerical results for $\frac{1}{K} \sum_{k=1}^K g^{[s]}(X_k)$ for $K=5000$, $K=10000$, $K=20000$, $K=50000$, and $K=100000$, in the cases $s=500$, $s=1000$, and $s=2000$ respectively, are tabulated in Figures 2, 3, 4 below.

\bigskip

\begin{figure}[hp] 
	
	\centering  
	
	\begin{tabular}{c|c|c|c|c|c} 
		
		\toprule 
		& $K=5000$  &  $   K=10000$   &  $   K=20000$   &   $K=50000 $ & $K=100000$   \\
		
		\midrule 
$E_1$&      3.4513 & 3.4541 & 3.45074 & 3.4417 & 3.44034 \\
 $E_2$&3.4208 & 3.43017 & 3.43423 & 3.43228 & 3.42744 \\
 $E_3$&3.4013 & 3.38951 & 3.40058 & 3.43055 & 3.42734 \\
 $E_4$&3.33751 & 3.36335 & 3.38431 & 3.40763 & 3.42225 \\
 $E_5$&3.27776 & 3.35011 & 3.40524 & 3.3932 & 3.40618 \\
 $E_6$&3.24535 & 3.33898 & 3.33186 & 3.4024 & 3.39838 \\
		\bottomrule 
		
	\end{tabular}
    \caption{Numerical results for $\frac{1}{K} \sum_{k=1}^K g^{[500]}(X_k)$}  
    \label{Numres}
\end{figure}

\bigskip

\begin{figure}[hp] 
	
	\centering  
	
	\begin{tabular}{c|c|c|c|c|c} 
		
		\toprule 
		& $K=5000$  &  $   K=10000$   &  $   K=20000$   &   $K=50000 $ & $K=100000$   \\
		
		\midrule 
      
 $E_1$&2.4422 & 2.44414 & 2.44149 & 2.43447 & 2.43337 \\
 $E_2$&2.41925 & 2.42568 & 2.42922 & 2.42732 & 2.42337 \\
 $E_3$&2.40409 & 2.39574 & 2.40298 & 2.426 & 2.42349 \\
 $E_4$&2.35616 & 2.37465 & 2.39071 & 2.40813 & 2.4195 \\
 $E_5$&2.31073 & 2.3628 & 2.40653 & 2.39708 & 2.40692 \\
 $E_6$&2.28672 & 2.35609 & 2.3498 & 2.40406 & 2.40116 \\
       
		\bottomrule 
		
	\end{tabular}
    \caption{Numerical results for $\frac{1}{K} \sum_{k=1}^K g^{[1000]}(X_k)$}  
    \label{Numres}
\end{figure}

\bigskip

\begin{figure}[hp] 
	
	\centering  
	
	\begin{tabular}{c|c|c|c|c|c} 
		
		\toprule 
		& $K=5000$  &  $   K=10000$   &  $   K=20000$   &   $K=50000 $ & $K=100000$   \\
		
		\midrule 
 
 $E_1$&1.72737 & 1.72919 & 1.72717 & 1.72179 & 1.72091 \\
 $E_2$&1.71064 & 1.71501 & 1.71777 & 1.71642 & 1.71329 \\
 $E_3$&1.69752 & 1.6939 & 1.69775 & 1.71531 & 1.71333 \\
 $E_4$&1.66527 & 1.6769 & 1.68915 & 1.70186 & 1.71047 \\
 $E_5$&1.63056 & 1.66685 & 1.70142 & 1.6933 & 1.70075 \\
 $E_6$&1.61208 & 1.66286 & 1.65744 & 1.69837 & 1.69624 \\       
		\bottomrule 
		
	\end{tabular}
    \caption{Numerical results for $\frac{1}{K} \sum_{k=1}^K g^{[2000]}(X_k)$}  
    \label{Numres}
\end{figure}

\bigskip

In a similar way for $s \geq 1$ define the function $h^{[s]}$ on $[0,1]^s$ given by:
\begin{eqnarray*}
& & h^{[s]}(u^{(1)},\cdots,u^{(s)}) \\ & =& 100 \cdot  \prod_{i=1}^s \cos\left(\frac{ \pi u^{(i) }}{2 i^{1/2}} \right)
\end{eqnarray*}

\bigskip

We have:
\begin{eqnarray}
& & \int_{[0,1]^s} h^{[s]} \, d \mu_{ST}^{[s]}  \\ &=& 100 \cdot \prod_{i=1}^s \int_{[0,1]} \cos\left(\frac{\pi u^{(i)}}{2 i^{1/2}}\right) \cdot 2 \sin^2(\pi u^{(i)})\, d u^{(i)}  \nonumber \\
  &=& 100 \cdot \prod_{i=1}^s  \left(           \frac{2 i^{1/2}}{\pi} \cdot \sin\left( \frac{\pi}{2 i^{1/2}}\right)  \cdot \frac{16 i}{16i-1}  \right)  \nonumber
\end{eqnarray}

\bigskip
The numerical values of $\int_{[0,1]^s} h^{[s]} \, d \mu_{ST}^{[s]}$ for $s=500$, $s=1000$, $s=1500$, and $s=2000$, are computed as in (2.3) to be:

\begin{eqnarray*}
\int_{[0,1]^{500}} h^{[500]} \, d \mu_{ST}^{[500]}  \stackrel{.}{=}  8.814
\end{eqnarray*}
\begin{eqnarray*}
\int_{[0,1]^{1000}} h^{[1000]} \, d \mu_{ST}^{[1000]}  \stackrel{.}{=} 6.92239
\end{eqnarray*}
\begin{eqnarray*}
\int_{[0,1]^{1500}} h^{[1500]} \, d \mu_{ST}^{[1500]}   \stackrel{.}{=}6.0099   
\end{eqnarray*}
\begin{eqnarray*}
\int_{[0,1]^{2000}} h^{[2000]} \, d \mu_{ST}^{[2000]}   \stackrel{.}{=}    5.43635
\end{eqnarray*}

\bigskip

While the numerical results for $\frac{1}{K} \sum_{k=1}^K h^{[s]}(X_k)$ for $K=5000$, $K=10000$, $K=20000$, $K=50000$, and $K=100000$, in the cases $s=500$, $s=1000$, $s=1500$, and $s=2000$ respectively, are tabulated in Figures 5, 6, 7, 8 below.

\bigskip
\bigskip

Summing up: the numerical results of this section supply evidences for the validity of Conjecture 1.1; this conjecture can be described as saying that, the sequence $\{x_k\}_{k \geq 1}$ is a pseudorandom sequence in $[0,1]$ with distribution law given by the Sato-Tate measure. 

\bigskip

\begin{figure}[hp] 
	
	\centering  
	
	\begin{tabular}{c|c|c|c|c|c} 
		
		\toprule 
		& $K=5000$  &  $   K=10000$   &  $   K=20000$   &   $K=50000 $ & $K=100000$   \\
		
		\midrule 
       $E_1$&$8.84564$&$ 8.8739$&$ 8.85475$&$ 8.82749$& $8.81591$\\
 $E_2$&$8.74331$&$ 8.7891$&$ 8.80022$&$ 8.78814$&$ 8.7736$\\
 $E_3$&$8.70232$&$ 8.63423$&$ 8.67626$&$ 8.79667$&$ 8.77114$\\
 $E_4$&$8.49599$&$ 8.54784$&$ 8.62823$&$ 8.69612$&$ 8.76489$\\
 $E_5$&$8.27357$&$ 8.53379$&$ 8.69698$&$ 8.67127$&$ 8.70012$\\
 $E_6$&$8.18894$&$ 8.50111$&$ 8.45486$&$ 8.70109$&$ 8.69085$\\
       
		\bottomrule 
		
	\end{tabular}
    \caption{Numerical results for $\frac{1}{K} \sum_{k=1}^K h^{[500]}(X_k)$}  
    \label{Numres}
\end{figure}

\bigskip
\bigskip

\begin{figure}[hp] 
	
	\centering  
	
	\begin{tabular}{c|c|c|c|c|c} 
		
		\toprule 
		& $K=5000$  &  $   K=10000$   &  $   K=20000$   &   $K=50000 $ & $K=100000$   \\
		
		\midrule 
$E_1$&$6.9513$&$ 6.97341$&$ 6.95763$&$ 6.93425$&$ 6.9241$\\
 $E_2$&$6.86542$&$ 6.90182$&$ 6.91247$&$ 6.90096$&$ 6.8879$\\
 $E_3$&$6.82891$&$ 6.77459$&$ 6.80619$&$ 6.90821$&$ 6.8863$\\
 $E_4$&$6.65655$&$ 6.69763$&$ 6.76587$&$ 6.822$&$ 6.88$\\
 $E_5$&$6.47181$&$ 6.67989$&$ 6.82385$&$ 6.80146$&$ 6.8251$\\
 $E_6$&$6.40304$&$ 6.65958$&$ 6.61765$&$ 6.82616$&$ 6.8182$\\
       
		\bottomrule 
		
	\end{tabular}
    \caption{Numerical results for $\frac{1}{K} \sum_{k=1}^K h^{[1000]}(X_k)$}  
    \label{Numres}
\end{figure}

\bigskip
\bigskip

\begin{figure}[hp] 
	
	\centering  
	
	\begin{tabular}{c|c|c|c|c|c} 
		
		\toprule 
		& $K=5000$  &  $   K=10000$   &  $   K=20000$   &   $K=50000 $ & $K=100000$   \\
		
		\midrule 
         $E_1$&$ 6.03675      $&$   6.0559         $&$     6.04206        $&$    6.02086         $&$   6.01151        $\\
 $E_2$&$   5.96016         $&$    5.99199         $&$   6.00129          $&$    5.99081         $&$   5.97872        $\\
 $E_3$&$     5.9214       $&$       5.87986      $&$    5.90473         $&$     5.9971        $&$     5.97718      $\\
 $E_4$&$     5.77546       $&$     5.80653        $&$   5.86934          $&$   5.91933          $&$   5.97276        $\\
 $E_5$&$     5.60953       $&$     5.7892        $&$     5.92279        $&$    5.90058         $&$    5.92183       $\\
 $E_6$&$     5.54555       $&$   5.77334          $&$    5.73411         $&$   5.92231          $&$   5.91537        $\\
       
		\bottomrule 
		
	\end{tabular}
    \caption{Numerical results for $\frac{1}{K} \sum_{k=1}^K h^{[1500]}(X_k)$}  
    \label{Numres}
\end{figure}

\bigskip
\bigskip

\begin{figure}[hp] 
	
	\centering  
	
	\begin{tabular}{c|c|c|c|c|c} 
		
		\toprule 
		& $K=5000$  &  $   K=10000$   &  $   K=20000$   &   $K=50000 $ & $K=100000$   \\
		
		\midrule 
       $E_1$&$5.46061$&$ 5.47935$&$ 5.46676$&$ 5.4467$&$ 5.4379$\\
 $E_2$&$5.39035$&$ 5.41982$&$ 5.42836$&$ 5.41881$&$ 5.4073$\\
 $E_3$&$5.3513$&$ 5.31787$&$ 5.33875$&$ 5.42458$&$ 5.4058$\\
 $E_4$&$5.2241$&$ 5.24905$&$ 5.30693$&$ 5.35259$&$ 5.4022$\\
 $E_5$&$5.06843$&$ 5.23096$&$ 5.35741$&$ 5.33487$&$ 5.3545$\\
 $E_6$&$5.01111$&$ 5.21909$&$ 5.18121$&$ 5.35458$&$ 5.3483$\\
       
		\bottomrule 
		
	\end{tabular}
    \caption{Numerical results for $\frac{1}{K} \sum_{k=1}^K h^{[2000]}(X_k)$}  
    \label{Numres}
\end{figure}

\section{Numerical Evidences for Conjecture 1.3, part I}

For numerics related to Conjecture 1.3 in the case $s=1$ (i.e. the original conjecture of Akiyama-Tanigawa), we refer to \cite{AT} and \cite{St}. To test Conjecture 1.3 directly one would need to evaluate the values the $D_K^{[s]}$ or $D_K^{*,[s]}$ for $K$ large. But in higher dimensions $s$ there are serious combinatorial difficulties in computing (even just numerically) the values of $D_K^{[s]}$ or $D_K^{*,[s]}$; this is the well known phenomenon known as the {\it Curse of Dimensionality}. 

\bigskip

We first note that, Conjecture 1.3 is obviously equivalent to the statement:
\[
\liminf_{K \rightarrow \infty} - \frac{\ln  D^{[s]}_K}{\ln K} \geq \frac{1}{2}
\]
(and similarly with $D_K^{[s]}$ being replaced by $D_K^{*,[s]}$).

\bigskip
\begin{proposition}
Let $f$ be a function defined on $[0,1]^s$, which is of bounded variation in the sense of Hardy and Krause. Then Conjecture 1.3 implies:
\[
\liminf_{K \rightarrow \infty} - \frac{\ln \big|       \frac{1}{K}\sum_{k=1}^K f(X_k)   - \int_{[0,1]^s} f \, d\mu_{ST}^{[s]}  \big|}{\ln K} \geq \frac{1}{2}
\]
\end{proposition}
\begin{proof}
This is an immediate consequence of the Koksma-Hlawka inequality ({\it c.f.} p. 151 of \cite{KN} and p. 967 of \cite{Ni2}):
\[
\Big|       \frac{1}{K}\sum_{k=1}^K f(X_k)   - \int_{[0,1]^s} f \, d\mu_{ST}^{[s]}  \Big|   \leq V(f) \cdot D_K^{*,[s]}
\]
where $V(f)$ is the total variation of $f$ in the sense of Hardy and Krause (the version of the Koksma-Hlawka inequality stated in {\it loc. cit.} is with respect to the Lebesgue measure on $[0,1]^s$, but the same proof works verbatim with respect to the measure $\mu_{ST}^{[s]}$ on $[0,1]^s$).
\end{proof}

\bigskip
In view of Proposition 3.1. we may then test Conjecture 1.3 indirectly as follows. With $f$ defined on $[0,1]^s$ as above (of bounded variation in the sense of Hardy and Krause), denote the relative error:

\[
\error(f,K) =    \frac{\frac{1}{K}\sum_{k=1}^K f(X_k)   - \int_{[0,1]^s} f \, d\mu_{ST}^{[s]} }{ \int_{[0,1]^s} f \, d\mu_{ST}^{[s]} }
\]

\bigskip
\noindent (assuming that the integral is nonzero). 

\bigskip

Then we evaluate:

\[
-\frac{\ln |\error (f,K)|}{\ln K}
\]
with $K$ being large. By virtue of Proposition 3.1, Conjecture 1.3 implies that

\[
\liminf_{K \rightarrow \infty}   \frac{-\ln |\error(f,K) |}{\ln K} \geq \frac{1}{2}
\]

\bigskip
\bigskip 

In the following numerical examples the dimensions $s$ and the test functions $f$ on $[0,1]^s$ are chosen as in section 2, namely: 

\[
f^{[10]},g^{[500]},g^{[1000]},g^{[2000]},h^{[500]},h^{[1000]},h^{[1500]}, h^{[2000]}
\]

\bigskip
The results are tabulated in Figures 9 -16 below.

\bigskip
\bigskip

\begin{figure}[hp] 
	
	\centering  
	
	\begin{tabular}{c|c|c|c|c|c} 
		
		\toprule 
		& $K=5 \times 10^5$  &  $   K=   10^6$   &  $   K= 2 \times 10^6$   &   $K=5 \times 10^6 $ & $K=10^7$   \\
		
		\midrule 
       $E_1$    &     0.69949&0.679386&0.630927&0.66101&0.691336 \\    
        $E_2$   &     0.633969&0.581946&0.687151&0.641671&0.670621\\

        $E_3$ &     0.573905&0.57178&0.585748&0.576866&0.592804\\

        $E_4$ &      0.523103&0.542605&0.531523&0.566703&0.525969\\

        $E_5$ &  0.514324&0.516659&0.513771&0.50841&0.504189\\
        $E_6$ &   0.494667&0.524282&0.505647&0.490742&0.491219\\
       
		\bottomrule 
		
	\end{tabular}
    \caption{Numerical results for $-\frac{\ln |\error (f^{[10]},K)|}{\ln K}$}  
    \label{Numres}
\end{figure}

\bigskip
\bigskip

\begin{figure}[hp] 
	
	\centering  
	
	\begin{tabular}{c|c|c|c|c|c} 
		
		\toprule 
		& $K=5 \times 10^5$  &  $   K=   10^6$   &  $   K= 2 \times 10^6$   &   $K=5 \times 10^6 $ & $K=10^7$   \\
		
		\midrule 
        $E_1$    & 0.638723& 0.648979& 0.631947& 0.862372& 0.808676          \\    
        $E_2$   &  0.575948&0.505878&0.633176&0.552754&0.543251     \\

        $E_3$ &  0.487867&0.496663&0.514921&0.501386&0.498532  \\

        $E_4$ & 0.444677&0.462&0.462372&0.493716&0.460063  \\

        $E_5$ &  0.431359&0.436573&0.438725&0.445434&0.44006      \\
        $E_6$ &   0.404953&0.436986&0.430817&0.421285&0.426719   \\
       
		\bottomrule 
		
	\end{tabular}
    \caption{Numerical results for $-\frac{\ln |\error (g^{[500]},K)|}{\ln K}$}  
    \label{Numres}
\end{figure}

\newpage

\begin{figure}[hp] 
	
	\centering  
	
	\begin{tabular}{c|c|c|c|c|c} 
		
		\toprule 
		& $K=5 \times 10^5$  &  $   K=   10^6$   &  $   K= 2 \times 10^6$   &   $K=5 \times 10^6 $ & $K=10^7$   \\
		
		\midrule 
        $E_1$    &  0.631169&0.645835&0.630361&0.712336&0.706233       \\    
        $E_2$   &   0.567706&0.497695&0.617072&0.544564&0.534925  \\

        $E_3$ &   0.480746&0.489362&0.507584&0.494107&0.492241   \\

        $E_4$ &   0.437471&0.455814&0.455769&0.487224&0.453975  \\

        $E_5$ &    0.424289&0.429748&0.432216&0.439139&0.434175  \\
        $E_6$ &    0.39818&0.430663&0.424525&0.414925&0.420545          \\
       
		\bottomrule 
		
	\end{tabular}
    \caption{Numerical results for $-\frac{\ln |\error (g^{[1000]},K)|}{\ln K}$}  
    \label{Numres}
\end{figure}

\bigskip
\bigskip
\bigskip
\bigskip
\bigskip

\begin{figure}[hp] 

	\centering  
	
	\begin{tabular}{c|c|c|c|c|c} 
		
		\toprule 
		& $K=5 \times 10^5$  &  $   K=   10^6$   &  $   K= 2 \times 10^6$   &   $K=5 \times 10^6 $ & $K=10^7$   \\
		
		\midrule 
            $E_1$    &   0.627366& 0.638355&0.622921&0.675519&0.66696        \\    
        $E_2$   &    0.559385&0.490252&0.605272&0.53634&0.526713  \\

        $E_3$ &      0.473656&0.482631&0.500549&0.487117&0.485837  \\

        $E_4$ & 0.430963&0.449776&0.449301&0.480237&0.447863\\

        $E_5$ &      0.418162&0.423543&0.426154&0.4331&0.428402  \\
        $E_6$ &      0.391767&0.424534&0.418375&0.408966&0.414847        \\
       
		\bottomrule 
		
	\end{tabular}
    \caption{Numerical results for $-\frac{\ln |\error (g^{[2000]},K)|}{\ln K}$}  
    \label{Numres}
\end{figure}

\bigskip
\bigskip
\bigskip
\bigskip
\bigskip

\begin{figure}[hp] 
	
	\centering  
	
	\begin{tabular}{c|c|c|c|c|c} 
		
		\toprule 
		& $K=5 \times 10^5$  &  $   K=   10^6$   &  $   K= 2 \times 10^6$   &   $K=5 \times 10^6 $ & $K=10^7$   \\
		
		\midrule 
        $E_1$    &   0.606447& 0.634717&0.596622&0.769728&0.72325 \\    
        $E_2$   &    0.592215&0.483297&0.643853&0.573784&0.538485    \\

        $E_3$ &      0.469397&0.46504&0.491604&0.482104&0.506944   \\

        $E_4$ &    0.428098&0.451091&0.44635&0.467237&0.44823 \\

        $E_5$ & 0.417505&0.421578&0.424573&0.429395&0.423905\\
        $E_6$ &  0.398343&0.426324&0.412744&0.399427&0.40276  \\
       
		\bottomrule 
		
	\end{tabular}
    \caption{Numerical results for $-\frac{\ln |\error (h^{[500]},K)|}{\ln K}$}  
    \label{Numres}
\end{figure}

\newpage

\begin{figure}[hp] 
	
	\centering  
	
	\begin{tabular}{c|c|c|c|c|c} 
		
		\toprule 
		& $K=5 \times 10^5$  &  $   K=   10^6$   &  $   K= 2 \times 10^6$   &   $K=5 \times 10^6 $ & $K=10^7$   \\
		
		\midrule 
        $E_1$    &   0.599239&0.632049&0.594434&0.699701&0.677617        \\   
        $E_2$   &    0.58362&0.4755&0.618896&0.561171&0.528059       \\

        $E_3$ &       0.462718&0.458196&0.484627&0.474634&0.499336 \\

        $E_4$ &    0.421188&0.445555&0.440249&0.46107&0.442045\\

        $E_5$ &   0.410704&0.414996&0.418179&0.423116&0.41825\\
        $E_6$ &  0.391969&0.420306&0.406692&0.393337&0.396837       \\
       
		\bottomrule 
		
	\end{tabular}
    \caption{Numerical results for $-\frac{\ln |\error (h^{[1000]},K)|}{\ln K}$}  
    \label{Numres}
\end{figure}

\bigskip
\bigskip

\begin{figure}[hp] 
	
	\centering  
	
	\begin{tabular}{c|c|c|c|c|c} 
		
		\toprule 
		& $K=5 \times 10^5$  &  $   K=   10^6$   &  $   K= 2 \times 10^6$   &   $K=5 \times 10^6 $ & $K=10^7$   \\
		
		\midrule 
        $E_1$  &  0.596406     & 0.624722         &      0.588404              &            0.675777              &    0.659381     \\    
        $E_2$   &   0.578855         &    0.471502                 &     0.612237                      &  0.555763       &   0.523419 \\

        $E_3$ &   0.458443            &   0.454467                &       0.480825                   &   0.470964      &    0.495892            \\

        $E_4$ &     0.41748           &    0.442316              &       0.436737                    &  0.457416        &      0.438645       \\

        $E_5$ &       0.40721        &    0.411428                &         0.414829               &     0.419795        &    0.415218      \\
        $E_6$ &    0.388235          &   0.416779                  &     0.403099                &     0.390014       &  0.393701         \\
       
		\bottomrule 
		
	\end{tabular}
    \caption{Numerical results for $-\frac{\ln |\error (h^{[1500]},K)|}{\ln K}$}  
    \label{Numres}
\end{figure}

\bigskip
\bigskip

\begin{figure}[hp] 
	
	\centering  
	
	\begin{tabular}{c|c|c|c|c|c} 
		
		\toprule 
		& $K=5 \times 10^5$  &  $   K=   10^6$   &  $   K= 2 \times 10^6$   &   $K=5 \times 10^6 $ & $K=10^7$   \\
		
		\midrule 
             $E_1$    &     0.594196&0.620986&0.584018&0.670973&0.654911       \\    
        $E_2$   &    0.574501&0.468545&0.607293&0.55214&0.520405   \\

        $E_3$ &     0.456099&0.452267&0.478601&0.468548&0.493642\\

        $E_4$ &    0.415082&0.440149&0.434365&0.45504&0.436376\\

        $E_5$ &   0.404872&0.409085&0.412494&0.417499&0.413082\\
        $E_6$ &  0.385872&0.414499&0.40085&0.387842&0.391689     \\
       
		\bottomrule 
		
	\end{tabular}
    \caption{Numerical results for $-\frac{\ln |\error (h^{[2000]},K)|}{\ln K}$}  
    \label{Numres}
\end{figure}

\section{Numerical Evidences for Conjecture 1.3, part II}

Finally we test Conjecture 1.3 directly for the dimensions $s=2$ and $s=3$. We first recall the following general result of Niederreiter \cite{Ni1} in order to compute the star discrepancy $D_K^{*,[s]}$ of $\{X_k\}_{k=1}^K \subset [0,1]^s$ (with respect to the measure $\mu_{ST}^{[s]}$).

\bigskip

For $1 \leq i \leq s$, denote by $0= \beta^{(i)}_1 < \cdots  <\beta^{(i)}_{n_i}=1$ the set of distinct values of the set of $i$-th coordinates of the points $\{X_k\}_{k=1}^K$, with the values $0$ and $1$ being included. Denote by $\mathfrak{q}$ the collection of rectangular regions $Q \subset [0,1]^s$ of the form:
\[
Q =  \prod_{i=1}^s (  \beta^{(i)}_{j_i}  , \beta^{(i)}_{j_i +1}    ],  \,\  1 \leq j_i < n_i \mbox{ for } 1 \leq i \leq s
\]
which thus forming a partition of $(0,1]^s$. For $Q$ as above, denote:
\[
Y(Q) =   (\beta^{(1)}_{j_1 +1},\cdots, \beta^{(s)}_{j_s +1}  )  \in [0,1]^s
\]
the upper end point of $Q$, and
\[
Z(Q) = (\beta^{(1)}_{j_1},\cdots, \beta^{(s)}_{j_s}  ) \in [0,1]^s  
\]
the lower end point of $Q$.

\bigskip
We then have:
\begin{proposition} The star discrepancy $D_K^{*,[s]}$ is equal to:
\begin{eqnarray*}
\max_{Q \in \mathfrak{q}} \Big(\max \Big(    \Big|   \frac{A(  \mathcal{R}_{Y(Q)};K)  }{K}     - \mu_{ST}^{[s]}(     \mathcal{R}_{Y(Q)}  )        \Big|,       \Big|          \frac{A(\mathcal{R}_{Z(Q)} ;K)}{K}     - \mu_{ST}^{[s]}(     \mathcal{R}_{Z(Q)}  )    \Big|         \Big)  \Big) 
\end{eqnarray*}
\end{proposition}
\begin{proof}
This is Theorem 2 of \cite{Ni1}. In {\it loc. cit.} it is stated with respect to the Lebesgue measure on $[0,1]^s$, but the argument works verbatim with respect to the measure $\mu_{ST}^{[s]}$.
\end{proof}

\bigskip
\begin{remark}
\end{remark}
\noindent To compute $D_K^{*,[s]}$ using Proposition 4.1 ({\it i.e.} Theorem 2 of \cite{Ni1}) requires the evaluation of  $O(K^s)$ terms. Although there are algorithms that improve upon that of \cite{Ni1} for computing the star discrepancy ({\it c.f.} for example \cite{DE}), the time complexity of the known algorithms is still exponential in terms of the dimension $s$; this is an instance of the {\it Curse of Dimensionality}. In fact it is known that the computation of star discrepancy belongs to the class of NP-hard problems \cite{GSW}.

\bigskip
\bigskip

Below we use Proposition 4.1 to compute the numerical values of $D_K^{*,[s]}$ and hence test Conjecture 1.3 (with respect to $D_K^{*,[s]}$), in the cases $s=2$ and $s=3$. The results are tabulated in Figures 17, 18 below.

\bigskip
Summing up: the numerical results of section 3 and section 4 supply evidences for the validity of Conjecture 1.3. It is a refinement of Conjecture 1.1, and can be regarded as a qualitative form, of the Law of Iterated Logarithm for random numbers ({\it c.f.} Chapter 7 of \cite{Ni3}). In particular, Conjecture 1.3 implies that, with respect to the Sato-Tate measure, the sequence of Frobenius angles of an elliptic curve over $\mathbf{Q}$ without complex multiplication forms a pseudorandom sequence in $[0,1]$ with strong randomness property, (at least) as far as statistical distribution is concerned.

\bigskip
\bigskip

\begin{figure}[hp] 
	
	\centering  
	
	\begin{tabular}{c|c|c|c|c|c} 
		
		\toprule 
		& $K=5 \times 10^3$  &  $   K=   10^4$   &  $   K= 2 \times 10^4$   &   $K=5 \times 10^4 $ & $K=10^5$   \\
		
		\midrule 
        $E_1$    &   0.513743         &     0.481735               &     0.493825                     &    0.508233   &  0.506597\\    
        $E_2$   &   0.506241         &     0.511887                &       0.468917                    &   0.494157      & 0.492688\\

        $E_3$ &    0.484667           &     0.483204              &       0.484577                   &   0.526097      & 0.487613\\

        $E_4$ &   0.442152             &    0.418237              &        0.434046                   &   0.440515       &0.467903\\

        $E_5$ &    0.423393           &    0.443996                &    0.46427                    &  0.441762           & 0.42277\\
        $E_6$ &    0.413569          &    0.421813                 &      0.419844               &  0.458051          &0.426849\\
       
		\bottomrule 
		
	\end{tabular}
    \caption{Numerical results for $-\frac{\ln  D_K^{*,[2]} }{\ln K}$}  
    \label{Numres}
\end{figure}

\bigskip
\bigskip

\begin{figure}[hp] 
	
	\centering  
	
	\begin{tabular}{c|c} 
		
		\toprule 
		& $K=5 \times 10^3$   \\
		
		\midrule 
        $E_1$   &      0.495306           \\    
        $E_2$   &     0.479892             \\

        $E_3$ &  0.472156                   \\

        $E_4$ &     0.413948               \\

        $E_5$ &          0.405477           \\
        $E_6$ &      0.39623                \\
       
		\bottomrule 
		
	\end{tabular}
    \caption{Numerical results for $-\frac{\ln  D_K^{*,[3]} }{\ln K}$}  
    \label{Numres}
\end{figure}

\section{Conclusion and Final Remarks}

In this paper we propose conjectures that refine the Sato-Tate conjecture, specifically we conjecture that the Frobenius angles of a given elliptic curve over $\mathbf{Q}$ without complex multiplication, are statistically independently distributed with respect to the Sato-Tate measure, including the more quantitative version involving the discrepancy of joint distributions. Numerical evidences are presented to support the conjectures. 

\bigskip

Taylor {\it et. al.} \cite{CHT,T,HSBT,BLGHT} had established the Sato-Tate conjecture for elliptic curves over totally real fields without complex multiplication, and more generally \cite{BLGG} established the Sato-Tate conjecture for Hilbert modular forms over totally real fields ({\it c.f.} \cite{ACC+} for the latest results on the Sato-Tate conjecture for automorphic forms over number fields). Thus it is natural to expect that our Conjectures 1.1 and 1.3 extend to the more general setting as well. 

\bigskip
It would be intriguing to find possible connections between Conjecture 1.3 in the case $s \geq 2$ and properties of $L$-functions.

\bigskip

Finally and most interestingly, as observed experimentally from the numerics, the rate of convergence to the measure $\mu_{ST}^{[s]}$, is slower in the case of curves with higher Mordell-Weil ranks (in the one dimensional case $s=1$ this was already observed in \cite{St}). Heuristically, in accordance with the original form of the Birch and Swinnerton-Dyer conjecture, this can be seen as due to the fact that, for curves of high Mordell-Weil rank, there is a Chebyshev bias for the quantities $a_{p_k}$ towards being negative, {\it c.f.} \cite{M}, \cite{S}, \cite{KM}. It would be important to understand the rate of convergence to the measure $\mu_{ST}^{[s]}$ in a more precise form (for example along the lines suggested in \cite{St} in the case $s=1$).

\end{document}